\documentclass{article}

\RequirePackage{amsthm,amsmath}
\usepackage{amsfonts}
\usepackage{amssymb}

\numberwithin{equation}{section}

\theoremstyle{plain}
\newtheorem{thm}{Theorem}[section]

\newtheorem{lem}[thm]{Lemma}
\newtheorem{prop}[thm]{Proposition}

\theoremstyle{definition}
\newtheorem{defn}[thm]{Definition}

\newtheorem{ex}[thm]{Example}

\newcommand{\abs}[1]{\left\vert#1\right\vert}
\newcommand{\set}[1]{\left\{#1\right\}}

\newcommand{\squ}[1]{\left[#1\right]}
\newcommand{\bra}[1]{(#1)} 
\newcommand{\norm}[1]{\left\Vert#1\right\Vert}
\newcommand{\prob}[1]{\operatorname{\mathbb{P}}(#1)}

\newcommand{\Ex}[1]{\operatorname{\mathbb{E}}\left[#1\right]}

\newcommand{\floor}[1]{\left\lfloor {#1} \right\rfloor}

\newcommand{\indev}[1]{\mathbf{1}_{\squ{#1}}}

\newcommand{\hyp}{\ensuremath{\mathbb{Z}_2^n}}

\newcommand{\RR}{\mathbb R}
\newcommand{\CC}{\mathbb C}
\newcommand{\distname}[1]{\mbox{\textit{#1}}}
\newcommand{\expdist}{\distname{Exp}}
\newcommand{\rinf}{\rightarrow\infty}
\newcommand{\rminf}{\rightarrow -\infty}

\newcommand{\rinz}{\rightarrow 0}
\newcommand{\nrinf}{\nrightarrow\infty}

\newcommand{\Ord}{\ensuremath{O}} 
\DeclareMathOperator{\separation}{sep}
\newcommand{\sept}[1]{\ensuremath{\separation \, (#1)}}
\newcommand{\septi}[1]{\ensuremath{\separation^i_n\, (#1)}}
\newcommand{\septn}[1]{\ensuremath{\separation_n\, (#1)}} 

\DeclareMathOperator*{\simlim}{\sim}

\newcommand{\ea}{\emph{et al.}}

\newcommand{\Xn}{\ensuremath{X_n}} 
\newcommand{\xn}{\ensuremath{x_n}} 
\newcommand{\Yn}{\ensuremath{Y_n}} 
\newcommand{\Tn}{\ensuremath{T_n}} 
\newcommand{\En}{\ensuremath{E_n}} 
\newcommand{\pin}{\ensuremath{\pi_n}} 
\newcommand{\lin}{\ensuremath{\lambda_n^i}} 

\begin{document}


\title{Separation and coupling cutoffs for tuples of independent Markov processes}

\author{Stephen Connor \\{\texttt sbc502@york.ac.uk} \\ \phantom{space} \\Department of Mathematics\\University of York\\
York YO10 5DD, UK}

\date{}

\maketitle

\begin{abstract}
\, We consider an $n$-tuple of independent ergodic Markov processes, each of which converges (in the sense of separation distance) at an exponential rate, and obtain a necessary and sufficient condition for the $n$-tuple to exhibit a separation cutoff. We also provide general bounds on the (asymmetric) window size of the cutoff, and indicate links to classical extreme value theory.
\end{abstract}


\section{Introduction}\label{sec:intro}

It is well known that a large number of Markov chains exhibit cutoff phenomena when converging to stationarity. This phenomenon occurs when the distance of the chain from equilibrium (measured using, for example, the total-variation metric or separation distance) stays close to its maximum value for some time, before dropping relatively fast and tending quickly to zero. Such behaviour was first identified for the transposition shuffle on the symmetric group \cite{Diaconis.Shahshahani-1981}, and has since been shown to hold for many natural sequences of random walks on groups (see
\cite{Saloff-Coste-2004} for a review).

In a recent paper~\cite{Barrera.Lachaud.ea-2006}, Barrera \ea\ consider $n$-tuples of independent processes, and give sufficient conditions for cutoffs to hold when distance from stationarity is measured using total-variation, Hellinger, chi-square and Kullback distances, under the assumption that each coordinate process converges exponentially fast. In the particular case when all coordinates converge at the same rate, the window size of the cutoff (to be defined below) is also determined.

In this paper we consider the separation distance of such $n$-tuples from stationarity and give conditions (very similar to those in \cite{Barrera.Lachaud.ea-2006}) guaranteeing the existence of a separation cutoff. Our approach is slightly different from that of Barrera \ea, however: instead of working with a set of ordered exponential rates we choose to work with discrete probability measures. This enables us to relate cutoff to convergence of (suitably scaled versions of) these measures. Furthermore, we are able to provide general bounds on the window size of the cutoff (not only when all coordinates converge at the same rate). In particular, we show that in general the right-hand side of the cutoff window may be of significantly larger order than the left.

The paper is organised as follows. In Section~\ref{sec:cutoff-phenomenon} we recall the definitions of total-variation and separation distance, and make formal the notion of cutoff time and window size.
In Section~\ref{sec:sep-cutoff} we present our main result concerning the existence of a separation cutoff, and prove general bounds on the window-size of such a cutoff. We then apply this to the example of a continuous-time random walk on the hypercube \hyp, where each coordinate may move at a different rate, and present a specific case which shows that our general window-size bounds are tight. Some links to classical extreme value theory are also highlighted. Finally, in Section~\ref{sec:coupling cutoff}, we briefly consider the notion of a \emph{coupling cutoff} for two such $n$-tuples.


\smallskip

\section{The cutoff phenomenon}\label{sec:cutoff-phenomenon}

In keeping with the notation of \cite{Diaconis.Saloff-Coste-2006}, for two probability measures $\mu$ and $\nu$ on a finite space $(E,\mathcal{E})$ we shall write $D(\mu,\nu)$ for a general notion of distance between them. One example is \emph{total-variation distance}
\[ D(\mu,\nu) = \norm{\mu - \nu}_{\textup{TV}} = \sup_{A\in\mathcal{E}} \abs{\mu(A) - \nu(A)} \,, \]
while the \emph{separation distance} is defined to be
\[ D(\mu,\nu) = \sept{\mu,\nu} = \max_{x\in E}\set{1- \frac{\mu(x)}{\nu(x)}} \,. \]
Note that separation is not a metric due to its asymmetry. Both of these distances take values in $[0,1]$, and it is simple to show \cite{Aldous.Diaconis-1987} that
\[ \norm{\mu - \nu}_{\textup{TV}} \leq \sept{\mu, \nu} \,. \] 

Separation distance is intimately linked with the notion of strong stationary times. Let $X$ be a Markov chain with time-$t$ distribution $P^t$ and stationary distribution $\pi$. 

\begin{defn}\label{defn:SST}
A \emph{strong stationary time (SST)} $T$ is a randomized stopping time for $X$ such that
\[ \prob{X_t = k\,|\, T\leq t} = \pi(k) \,, \qquad\text{for all $0\leq t<\infty$, $k\in E$.} \]
\end{defn}
If $T$ is a SST for $X$, then~\cite{Aldous.Diaconis-1987}
\begin{equation}\label{eqn:SST-inequality}
\sept{P^t, \pi} \leq \prob{T>t}\,, \qquad \text{for all $t\geq 0$.} 
\end{equation}
An optimal SST is one which achieves equality in (\ref{eqn:SST-inequality}) for all $t\geq 0$: existence is demonstrated in~\cite{Aldous.Diaconis-1987} (in discrete-time).

\medskip

We may now define the notion of a cutoff phenomenon for a given distance function $D$ (including, but not restricted to, those distances defined above).

\begin{defn}\label{defn:cutoff}
For $n\geq 1$, let $\Xn$ be a stochastic process taking values on a finite space $(\En,\mathcal{E}_n)$, with time-$t$ distribution $P_n^t$ and stationary distribution $\pin$. We say that the sequence
$\set{\En, \Xn, \pin \,;\, n=1,2,\dots}$ exhibits a \emph{$(\tau_n,b_n)$-$D$-cutoff} if $\tau_n,b_n>0$ satisfy
$b_n=o(\tau_n)$ and
\begin{align*}
d_-(c) = \liminf_{n\rinf} D\bra{P_n^{\tau_n+cb_n},\pi_n}
\quad&\text{satisfies $\lim_{c\rminf}d_-(c) = 1$} \,, \\
d_+(c) = \limsup_{n\rinf} D\bra{P_n^{\tau_n+cb_n},\pi_n}
\quad&\text{satisfies $\lim_{c\rinf}d_+(c) = 0$}\,.
\end{align*}
\end{defn}

Here $\tau_n$ is called the \emph{cutoff time}, and $b_n$ will be referred
to as the \emph{window} of the cutoff. (We may simply say that the sequence $\Xn$ exhibits a $\tau_n$-$D$-cutoff when we are not concerned with the window size $b_n$.) 

Furthermore, it is possible to analyse the window size in more detail by
considering separately the windows either side of the cutoff time
$\tau_n$. That is, instead of using a single sequence $b_n$ to
establish convergence in equations~\eqref{eqn:F-inf-conv} and
\eqref{eqn:F-sup-conv}, we can consider each convergence statement
separately.
\begin{defn}\label{defn:left/right-windows}
Suppose the sequence $\set{\En, \Xn, \pin}$ exhibits a
$\tau_n$-$D$-cutoff. If there exist sequences $b_n^L$, $b_n^R$
with $\max\{b_n^L,b_n^R\}=o(\tau_n)$, such that %
\begin{align*}
d^L_-(c) &= \liminf_{n\rinf} D\bra{P_n^{\tau_n+cb_n^L},\pin}
\quad\text{satisfies $\lim_{c\rminf}d^L_-(c) = 1$,} \\
\text{and}\quad d^R_+(c) &= \limsup_{n\rinf} D\bra{P_n^{\tau_n+cb_n^R},\pin}
\quad\text{satisfies $\lim_{c\rinf}d^R_+(c) = 0$,}
\end{align*}
then $b_n^L$ will be called a \emph{left-window} and $b_n^R$ a \emph{right-window} of the cutoff.
\end{defn}

To the best of the author's knowledge, the only published article to identify a difference between the left and right windows of a cutoff phenomenon is
~\cite{Chen2008}. For the processes considered in this paper however, such a distinction will prove to be rather important.



\section{Separation cutoff for $n$-tuples of independent processes}\label{sec:sep-cutoff}

Let $\Xn = (X_n^1,X_n^2,\dots,X_n^n)$ be an $n$-tuple of independent, continuous-time Markov chains on a finite space $(E_n,\mathcal{E}_n)$, with initial state $x_n = (x_n^1,\dots,x_n^n)$ and stationary distribution $\pin=\pin^1\times\ldots\times\pin^n$. Let 
\[ \septn{t} = \sept{P_n^t\,,\pin} \quad\text{and}\quad \septi{t} = \sept{P_{n,i}^t\,, \pin^i} \,, \]
where $P_{n,i}^t$ denotes the distribution of $\Xn^i$ at time $t$.

\begin{prop}\label{prop:from_one_to_many}
For all $t\geq 0$, 
\[ \septn t = 1- \prod_{i=1}^n (1-\septi t) \,. \]
\end{prop}

\begin{proof}
The independence of the chains implies that
\[ 1-\septn t = \min_{y_n^1,\dots,y_n^n} \prod_{i=1}^n \frac{P_{n,i}^t(x_n^i,y_n^i)}{\pin^i(y_n^i)}  = \prod_{i=1}^n (1-\septi t)\,, \]
since each term in the product may be minimised individually.
\end{proof}

If $T_n^i$ is an optimal SST for $\Xn^i$ ($1\leq i\leq n$), then letting $T_n = \max T_n^i$ one can check that $T_n$ is a SST for the $n$-tuple. Proposition~\ref{prop:from_one_to_many} shows that
\[ \septn t = \prob{T_n >t} \quad \text{for all $t\geq 0$,} \]
and it follows that $T_n$ is an optimal SST for $\Xn$.


\medskip
As in~\cite{Barrera.Lachaud.ea-2006}, we are interested in processes for which each component $X_n^i$ converges at an exponential rate $\lambda_n^i$, although now this convergence is to be measured using separation distance. Rather than following the route of~\cite{Barrera.Lachaud.ea-2006} and using an ordered set of rates $\{\lambda_{(i,n)}\}$, we prefer to work instead with discrete probability measures $\mu_n$ on $(0,\infty)$, where 
\[ \mu_n(\set{\lambda}) = \frac{1}{n}\,\#\{\lambda_n^i \,:\, \lambda_n^i = \lambda\} \,. \]
(This is similar to the use of \emph{design measures} in design theory, see e.g. \cite{St.John.Draper-1975}.) The
result of this will be that the existence of a separation cutoff can be
directly related to the convergence of appropriately scaled versions
of $\mu_n$ as $n\rinf$. We define $\kappa_n$ by
\[ \kappa_n = \min\set{\lambda>0\,:\, \mu_n(0,\lambda]>0} \,. \]
The main result of this paper is the following:

\begin{thm}\label{thm:main-prelim}
Let $\Xn$ be an $n$-tuple of independent ergodic Markov processes, each of whose components satisfies $|g_{\lin}(t)|\leq g(t)$ for all $t\geq0$, where $g_{\lin}$ is defined by
    \[ \frac{\log \septi t}{t} + \lin = g_{\lin}(t) \,, \]
and where $g$ is a bounded continuous function satisfying $g(t)\leq \Ord(t^{-1})$. As above, let $\mu_n$ be the discrete probability measure describing the set $\{\lin\}$, with support $[\kappa_n,\infty)$.
\begin{enumerate}
\item The sequence of $n$-tuples $\Xn$ exhibits a separation cutoff at time
        \[ \tau_n = \max_{\lambda\geq \kappa_n} \set{\frac{\log
(n\mu_n(0,\lambda])}{\lambda}}
 \]
if and only if $\tau_n\kappa_n\rinf$;
\item The window of the separation cutoff is in general asymmetric: the left side is at most $\Ord(1/\kappa_n)$, and the right side is bounded above by $W(\tau_n\kappa_n)/\kappa_n$, where $W$ is the Lambert $W$-function.
\end{enumerate}
\end{thm}

As remarked in \cite{Barrera.Lachaud.ea-2006}, under the conditions of Theorem~\ref{thm:main-prelim} the spectral gap of $\Xn$ is equal to $\kappa_n$ and the separation-mixing time equivalent to $\tau_n$. Thus Theorem~\ref{thm:main-prelim}(i) shows that the conjecture of Peres (reported in \cite{Diaconis.Saloff-Coste-2006, Chen2008}) holds true for separation cutoff for the processes considered here.

\medskip
Consider an $n$-tuple $\Xn$ satisfying the conditions of Theorem~\ref{thm:main-prelim}. Using $\mu_n$ and Proposition~\ref{prop:from_one_to_many}, the separation distance at time $t$ may be written as
\begin{equation}\label{eqn:sep-measure}
\septn{t} = 1-\exp \left(n \int_{\kappa_n}^\infty
\log\bra{1-e^{-t(\lambda - g_\lambda(t))}}\mu_n(d\lambda)\right) \,.
\end{equation}
One benefit of working with separation distance in this setting is that equation~\eqref{eqn:sep-measure} holds for any $\mu_n$, whereas there is no longer a simple exact expression for the total-variation distance between $P_n^t$ and $\pin$ when the rates $\lambda_n^i$ are not identical \cite{Barrera.Lachaud.ea-2006}.

The proof of Theorem~\ref{thm:main-prelim} will be established by the results of Proposition~\ref{prop:nec-condition}, Lemma~\ref{lem:coupling cutoff-window} and Theorem~\ref{thm:coupling cutoff-window} below.

\begin{prop}\label{prop:nec-condition}
For the sequence $\{\Xn\}$ to exhibit a $\tau_n$-separation cutoff, it is necessary for $\tau_n\kappa_n\rinf$. 
\end{prop}
\begin{proof}
Restricting attention to the mass at $\kappa_n$ in equation~\eqref{eqn:sep-measure} immediately implies that, for any $c>1$,
\begin{align*}
\septn{c\tau_n} &\geq \exp\bra{-c\tau_n(\kappa_n-g_{\kappa_n}(c\tau_n))} \\
& \geq \exp\bra{-c\tau_n\kappa_n}\exp\bra{-c\tau_n g(c\tau_n)} \,.
\end{align*}
For a separation cutoff to hold at $\tau_n$, we require that $\septn{c\tau_n}\rinz$ for all fixed $c>1$: this fails, however, if $\tau_n{\kappa_n}\nrinf$ (since the final exponential term above is bounded away from zero due to our conditions on $g$).
\end{proof}



\medskip

Now, given a measure $\mu_n$, define $\tau_n$ by
\begin{equation}\label{eqn:sep-cutoff-time}
\tau_n = \max_{\lambda\geq{\kappa_n}} \set{\frac{\log
(n\mu_n(0,\lambda])}{\lambda}} = \frac{\log
(n\mu_n(0,\lambda_n^*])}{\lambda_n^*} \,,
\end{equation}
where $\lambda_n^*\in[\kappa_n,\infty)$ is defined by this last equality.
(If there are two or more values of $\lambda$ achieving the maximum
in equation~\eqref{eqn:sep-cutoff-time} then we shall
(arbitrarily) always take $\lambda_n^*$ to be the minimum of these
values.)  Given $\lambda_n^*$, we may define a new measure $\nu_n$ on
$(0,\infty)$ as follows:
\begin{equation}\label{eqn:defn-v}
\nu_n (\{x\}) =
\frac{\mu_n(\{\lambda_n^*\,x\})}{\mu_n(0,\lambda_n^*]} \,.
\end{equation}
This measure has total mass
$\bra{\mu_n(0,\lambda_n^*]}^{-1}\in[1,n]$ and satisfies
$\nu_n(0,1] = 1$. The idea behind this scaling is as follows:
$\lambda_n^*$ describes in some sense the `critical point' of
$\mu_n$ -- it will be shown that if $\tau_n{\kappa_n}\rinf$ then any mass
$\mu_n$ places to the left of $\lambda_n^*$ will not influence the
separation cutoff time. For ease of notation we define
\[ \beta_n = n\mu_n(0,\lambda_n^*] \in[1,n] \,. \]

\smallskip

\begin{lem}\label{lem:v[0,1]->delta1}
If $\tau_n{\kappa_n}\rinf$ then:
\begin{enumerate}
\item[(i)] $\beta_n\rinf$;
\item[(ii)] $\nu_n(0,1]\xrightarrow{w}\delta_1$ (where $\xrightarrow{w}$ denotes weak convergence).
\end{enumerate}
\end{lem}

\begin{proof} 
\begin{enumerate}
\item[(i)] $\beta_n=\exp(\tau_n\lambda_n^*) \geq \exp(\tau_n\kappa_n) \rinf$ by assumption.
\item[(ii)] By definition of $\tau_n$
\eqref{eqn:sep-cutoff-time},
\[ \frac{\log\bra{n\mu_n(0,\lambda]}}{\lambda} \leq
\frac{\log\beta_n}{\lambda_n^*} \quad\text{for all $\lambda\geq\kappa_n$.}
\]  Thus for all $x\geq \kappa_n/\lambda_n^*$,
\[ \frac{\log\bra{n\mu_n(0,x\lambda_n^*]}}{x} \leq \log\beta_n \,. \]
This yields
\begin{equation}\label{eqn:prop-proof1}
n\mu_n(0,x\lambda_n^*] \leq \beta_n^x \quad\text{for all $x\geq
\kappa_n/\lambda_n^*$.}
\end{equation}
Hence
\begin{equation}\label{eqn:bound-on-v}
\nu_n(0,x] = \frac{\mu_n(0,x\lambda_n^*]}{\mu_n(0,\lambda_n^*]} =
\frac{n\mu_n(0,x\lambda_n^*]}{\beta_n} \leq \beta_n^{x-1} \,,
\end{equation}
where the inequality follows from \eqref{eqn:prop-proof1}. Thus for
all $\varepsilon\in(0,1)$,
\[ \nu_n(0,1-\varepsilon] \leq \beta_n^{-\varepsilon}
\xrightarrow[n\rinf]{}0 \,. \] 
Since $\nu_n(0,1]=1$ for all $n$, this proves the required convergence.
\end{enumerate}
\end{proof}

This makes more precise what is meant by $\lambda_n^*$ describing
the `critical point' of $\mu_n$. Under the assumption that
$\tau_n{\kappa_n}\rinf$, the measures $\nu_n$ converge weakly to $\delta_1$ on $(0,1]$:
this is exactly the sort of behaviour to be expected if the sequence
$\set{\lambda_n^*}$ captures information about the cutoff
time. Lemma~\ref{lem:coupling cutoff-window} and Theorem~\ref{thm:coupling cutoff-window} make this observation exact: their proofs rely on Proposition~\ref{prop:theta-behaviour}, which describes the behaviour of the function $\theta_n$ defined by
\begin{equation}\label{eqn:defn-theta}
     \theta_n(t) = \beta_n \int_{{\kappa_n}/\lambda_n^*}^\infty
        \exp\bra{-t\lambda_n^* \lambda} \nu_n(d\lambda) \,.
\end{equation}

\smallskip

\begin{prop}\label{prop:theta-behaviour}
The following inequalities hold for all $t\geq \log2/\kappa_n$:
\begin{equation}\label{eqn:theta-bounds-F-2}
1-\exp\bra{-e^{-tg(t)}\theta_n(t)} \leq \septn{t} \leq
1-\exp\bra{-2e^{2tg(t)}\theta_n(t)} \,.
\end{equation}
\end{prop}
Note that if $\tau_n\kappa_n\rinf$, Proposition~\ref{prop:theta-behaviour} implies that the behaviour of $\textup{sep}^{(n)}$ around $\tau_n$ is determined by that of $\theta_n$.
\begin{proof}
Using the measure $\nu_n$, the separation in equation~\eqref{eqn:sep-measure} may be rewritten as follows:
\begin{equation}\label{eqn:sep-using-nu}
\septn{t} = 1- \exp\left(\beta_n \int_{{\kappa_n}/\lambda_n^*}^\infty
\log\left(1-e^{-t(\lambda_n^*\lambda - g_{\lambda_n^*\lambda}(t))}\right)\nu_n(d\lambda)\right) \,.
\end{equation}
Now note that the following simple inequality holds for $0\leq x\leq 1/2$:
\begin{equation*}\label{eqn:basic-inequality}
-x-x^2 \leq \log(1-x) \leq -x \,.
\end{equation*}
Applying this inequality to the $\log$ term in equation~\eqref{eqn:sep-using-nu}, and bounding $g_{\lambda_n^*\lambda}(t)$ by $\pm g(t)$, shows that for all $t\geq \log2/\kappa_n$:
\begin{equation*}\label{eqn:theta-bounds-F-1}
1-\exp\bra{-e^{-tg(t)}\theta_n(t)}\leq \septn{t} \leq
1-\exp\bra{-e^{tg(t)}\theta_n(t) - e^{2tg(t)}\theta_n(2t)}  \,.
\end{equation*}
Finally, observe from \eqref{eqn:defn-theta} that $\theta_n(2t)\leq\theta_n(t)$ for all $t\geq 0$: the result follows immediately.
\end{proof}

We are now in a position to prove the existence of the left-hand side of the cutoff in Theorem~\ref{thm:main-prelim}.

\begin{lem}\label{lem:coupling cutoff-window}
Suppose that $\tau_n{\kappa_n}\rinf$, with $\tau_n$ defined as in (\ref{eqn:sep-cutoff-time}). Let $b_n^L=1/\lambda_n^*\leq \Ord(1/\kappa_n)$. Then
\[ \textup{sep}^L_-(c) = \liminf_{n\rinf}\, \septn{\tau_n+cb_n^L}
\quad\text{satisfies $\lim_{c\rminf}\textup{sep}^L_-(c) = 1$.} \]

\end{lem}

(Note that since $\tau_n\kappa_n\rinf$, $b_n^L=o(\tau_n)$, as is required for any candidate window-size.)

\begin{proof}
Consider
$\theta_n(\tau_n+c/\lambda_n^*)$, for fixed $c\in\RR$. Since
$\tau_n{\kappa_n}\rinf$ it follows from Lemma~\ref{lem:v[0,1]->delta1}(i) that for any fixed $c\in\RR$, 
\[ \tau_n+\frac{c}{\lambda_n^*} = \frac{\log\beta_n+c}{\lambda_n^*}\geq 0 \]
for large enough $n$. By definition of $\tau_n$, with
$\tau_n+c/\lambda_n^*\geq 0$:
\begin{align}
 \theta_n(\tau_n+c/\lambda_n^*) &=
\beta_n \int_{{\kappa_n}/\lambda_n^*}^\infty
\exp\bra{-\squ{\tau_n+c/\lambda_n^*}\lambda_n^*\lambda}
\nu_n(d\lambda) \nonumber \\
&\geq \beta_n \int_{{\kappa_n}/\lambda_n^*}^1
\exp\bra{-\squ{\tau_n+c/\lambda_n^*}\lambda_n^*\lambda} \nu_n(d\lambda) \nonumber \\
&\geq \beta_n \nu_n(0,1] \left(\frac{e^{-c}}{\beta_n}\right) = e^{-c} \,.
\label{eqn:theta-lower-bound}
\end{align}

Combining Proposition~\ref{prop:theta-behaviour} and inequality~\eqref{eqn:theta-lower-bound} shows that for all $c\in\RR$,
\[ \textup{sep}^L_-(c) \geq 1-\limsup_{n\rinf}\,
\exp\bra{-e^{-\gamma_n^L(c)}\theta_n(\tau_n+c/\lambda_n^*)} \,, \] 
where  
\begin{equation}\label{eqn:gamma}
\gamma_n^L(c) = \bra{\tau_n+cb_n^L}g\bra{\tau_n+cb_n^L} \simlim_{n\rinf} \tau_n g(\tau_n)= \Ord(1)  \,.
\end{equation}
Hence
\[ \textup{sep}^L_-(c) \geq 1-\exp\bra{-Me^{-c}} \,, \]
for some finite constant $M>0$, and thus $\textup{sep}^L_-(c)\rightarrow 1$ as $c\rminf$, as claimed.
\end{proof}

It turns out that the general bound for the
right-window of the cutoff is significantly larger than
that for the left. Theorem~\ref{thm:coupling cutoff-window}, which completes the proof of Theorem~\ref{thm:main-prelim}, makes use of the Lambert
$W$-function (see \cite{Corless.Gonnet.ea-1996}). This is the function defined for all $x\in\CC$ by%
\[ W(x) e^{W(x)} = x \,. \]
$W(x)$ is positive and increasing for $x\in\RR^+$, with $\displaystyle{W(x) \sim \log (x/\log x)}$ as $x\rinf$.

\begin{thm}\label{thm:coupling cutoff-window}
Suppose that $\tau_n{\kappa_n}\rinf$, with $\tau_n$ defined as in (\ref{eqn:sep-cutoff-time}). Then
\[ \textup{sep}^R_+(c) = \limsup_{n\rinf}\, \septn{\tau_n+cW(\tau_n\kappa_n)/\kappa_n}
\quad\text{satisfies $\lim_{c\rinf}\textup{sep}^R_+(c) = 0$.} \]
\end{thm}
\begin{proof}
In order for a sequence $b_n^R$ to be a right-window for the separation cutoff,
it is sufficient to show that $\theta_n\bra{\tau_n+cb_n^R}\leq h(c)$
for sufficiently large $n$, where $h(c)\rinz$ as $c\rinf$. For then,
using inequality~\eqref{eqn:theta-bounds-F-2} it follows that
\begin{align*}
\textup{sep}^R_+(c) &= \limsup_{n\rinf}\, \septn{\tau_n+cb_n^R} \\
&\leq 1- \liminf_{n\rinf}\, \exp\bra{-2e^{2\gamma_n^R(c)}\theta_n\bra{\tau_n+cb_n^R}} \,,
\end{align*}
where $\gamma_n^R(c)$ is defined analogously to~(\ref{eqn:gamma}). Thus, for some finite $M$, 
\[ \textup{sep}^R_+(c) \leq 1- \exp\bra{-Mh(c)} \xrightarrow[c\rinf]{} 0\,. \]

We therefore search for an upper bound on the function
$\theta_n\bra{\tau_n+cb_n^R}$ for fixed $c>0$. The form of $\tau_n$, and use of integration by parts, yield the
following:
\begin{align}
\theta_n\bra{\tau_n+cb_n^R} &= \beta_n \int_{{\kappa_n}/\lambda_n^*}^\infty
\left(\frac{e^{-cb_n^R\lambda_n^*}}{\beta_n}\right)^\lambda
\nu_n(d\lambda) \nonumber \\
&= \beta_n\squ{\left(\frac{e^{-cb_n^R\lambda_n^*}}{\beta_n}\right)^\lambda
\nu_n(0,\lambda]}^\infty_{\kappa_n/\lambda_n^*}  \nonumber \\
 &\qquad + \beta_n \log(\beta_n e^{cb_n^R\lambda_n^*})
\int_{{\kappa_n}/\lambda_n^*}^\infty
\left(\frac{e^{-cb_n^R\lambda_n^*}}{\beta_n}\right)^\lambda
\nu_n(0,\lambda] \,d\lambda \,. \label{eqn:int-by-parts}
\end{align}
Now, for $c>0$, this first term is negative for all $n$. Discarding
this, and using inequality~\eqref{eqn:bound-on-v} to bound
$\nu_n(0,\lambda]$ in the second term, we see that
\begin{align}
\theta_n\bra{\tau_n+cb_n^R} &\leq \beta_n \log(\beta_n
e^{cb_n^R\lambda_n^*}) \int_{{\kappa_n}/\lambda_n^*}^\infty
\left(\frac{e^{-cb_n^R\lambda_n^*}}{\beta_n}\right)^\lambda
\beta_n^{\lambda-1} \,d\lambda \nonumber \\
 &= \log(\beta_n e^{cb_n^R\lambda_n^*})
 \frac{e^{-cb_n^R{\kappa_n}}}{cb_n^R\lambda_n^*} \nonumber \\
 &= e^{-cb_n^R{\kappa_n}}\left(\frac{\tau_n}{cb_n^R}+1\right) \,, \quad\text{by definition of $\tau_n$.} \label{eqn:right-window-theta1}
\end{align}
Since $b_n^R$ must satisfy $b_n^R=o(\tau_n)$, 
this upper bound tends
to infinity with $n$ unless $cb_n^R{\kappa_n}\geq W(\tau_n{\kappa_n})$, by definition of the Lambert $W$-function. Thus, with
$b_n^R=W(\tau_n{\kappa_n})/{\kappa_n}$, \eqref{eqn:right-window-theta1} satisfies
\[ e^{-cb_n^R\kappa_n}\left(\frac{\tau_n}{cb_n^R}+1\right) \xrightarrow[n\rinf]{} h(c) =
\begin{cases}
\infty &\quad 0<c<1 \\
1 &\quad c=1 \\
0 &\quad c>1 \,.
\end{cases}
\]
It follows that for $c> 1$, $\theta_n(\tau_n+cW(\tau_n\kappa_n)/\kappa_n)
\rinz$ as $n\rinf$, and so
\[ \textup{sep}^R_+(c) = \limsup_{n\rinf}\, \septn{\tau_n+cW(\tau_n\kappa_n)/\kappa_n} =0 \,. \]
Therefore $b_n^R=W(\tau_n\kappa_n)/\kappa_n$ is a right-window of the cutoff, as claimed.
\end{proof}

This bound on the right-window is significantly larger than that for
the left-window. Since $\tau_n\kappa_n$ necessarily tends to infinity when a
separation cutoff holds, it follows that 
\[ \Ord(1/\kappa_n) < \frac{W(\tau_n\kappa_n)}{\kappa_n} = o(\tau_n) \,. \]



\subsection{Random walks on \hyp}\label{ssec:hypercube}

Let \hyp\ be the group of binary $n$-tuples under coordinate-wise
addition modulo 2: this can be viewed as the vertices of an $n$-dimensional
hypercube. A continuous-time random walk $\Xn$ on \hyp\ may be defined as follows. Let $\{\Lambda_n^i \,: \,1\leq i\leq n \}$ be a set of independent Poisson processes, with the rate of $\Lambda_n^i$ equal to $2\rho_n^i$: whenever there is an incident on $\Lambda_n^i$, with probability $1/2$ the $i^{th}$ coordinate, $X_{n,i}$, is flipped to its opposite value. The unique equilibrium distribution of $\Xn$ is the uniform distribution on \hyp, $U_n$. 

Let $T_n^i$ be the time of the first incident on $\Lambda_n^i$. It is simple to show that $T_n^i$ is an optimal SST for $X_{n,i}$, with
\[ \prob{T_n^i>t} = e^{-2t\rho_n^i} \,. \]
Thus $T_n = \max T_n^i$ is an optimal SST for $\Xn$. (This is similar in flavour to the optimal SST for the continuous-time birth-death processes of \cite{Diaconis.Saloff-Coste-2006}: there the SST is given by a sum of exponential random variables of varying rates, rather than their maximum.)

It follows that $X_n$ satisfies the conditions of Theorem~\ref{thm:main-prelim}, with 
\[ \lambda_n^i = 2\rho_n^i\,, \quad\text{and}\quad g\equiv 0 \,. \]
Writing $\rho_n^* = \min\set{\rho_n^i}$, the sequence $X_n$ therefore exhibits a separation cutoff at time
\[ \tau_n= \max_{\rho\geq \rho_n^*}\set{\frac{\log(n\mu_n(0,2\rho])}{2\rho}} \]
if and only if $\tau_n\rho_n^*\rinf$. In this case, $\tau_n = 2\hat\tau_n$, where $\hat\tau_n$ is the total-variation cutoff time according to Theorem 12 of~\cite{Barrera.Lachaud.ea-2006}.



\medskip

For many simple examples, such as the symmetric random
walk for which all coordinates jump at rate 1, the result of Theorem~\ref{thm:coupling cutoff-window} gives an extremely conservative bound for the right-window. (Simple direct calculations show that a $(\log n/2,1)$-separation cutoff holds, whereas the bound on $b_n^R$ from Theorem~\ref{thm:coupling cutoff-window} tends to infinity with $n$.) However, the following example shows that the bound of Theorem~\ref{thm:coupling cutoff-window} can be achieved, and so
cannot be improved upon in general.

\begin{ex}\label{ex:odd-windows}
Consider the sequence of random walks on \hyp\ with $\rho_n^i = \max\set{1, 2\log_n(i)}$. The associated probability measure for $X_n$ is
\[ \mu_n = \frac{1}{n} \sum_{i=1}^n \delta_{\max\set{2,\,4\log_n(i)}} \,. \]
The measure $\mu_n$ places all its mass in the
interval $[2,4]$, with $\kappa_n=2$ and 
\[ \mu_n[2,\lambda] = \frac{\floor{n^{\lambda/4}}}{n} \sim
n^{\lambda/4 - 1}, \quad\text{for all $\lambda\in[2,4]$.}
\]
For this sequence,
\[ \tau_n = \max_{2\leq\lambda\leq 4} \set{\frac{\log\bra{n
\mu_n[2,\lambda]}}{\lambda}} = \max_{2\leq\lambda\leq 4}
\set{\frac{\log\bra{n^{\lambda/4}}}{\lambda}} = \frac{\log n}{4}
\,. \]
 Note that this maximum is attained at all $\lambda\in[2,4]$:
we arbitrarily take $\lambda_n^*=2$ to be the minimum of these values.
This gives $\beta_n = \sqrt{n}$, and hence $\nu_n[1,x] =
n^{(x-1)/2}$ for $x\in[1,2]$. Since $\tau_n\rinf$ as $n\rinf$, this
random walk exhibits a $\tau_n$-separation cutoff.
Now, by Lemma~\ref{lem:coupling cutoff-window}, the
left-window of this separation cutoff is bounded above by
$1/\lambda_n^* = 1/2$.
However, for fixed $c>0$ and
some sequence $b_n^R=o(\tau_n)$, integration by parts as in
equation~\eqref{eqn:int-by-parts} yields the following:
\begin{align*}
\theta_n\left(\frac{\log n}{4}+cb_n^R\right) &\sim \bra{e^{-4cb_n^R}- e^{-2cb_n^R}} \\
&\qquad  + \sqrt{n} \log(\sqrt{n}
\,e^{2cb_n^R}) \int_1^2 \left(\frac{e^{-2cb_n^R}}{\sqrt{n}}\right)^\lambda
n^{(\lambda-1)/2} \,d\lambda \\
&\sim e^{-2cb_n^R} \frac{\tau_n}{cb_n^R}\,.
\end{align*}
Arguing as in the proof of Theorem~\ref{thm:coupling cutoff-window}, a $(\tau_n,b_n^R)$-separation cutoff does not hold for any sequence $b_n^R =o(W(\tau_n))$ (see \cite{Connor-2007} for further details).
\end{ex}


\subsection{Links to extreme value theory}\label{ssec:EVT}

Looking back to the discussion following Proposition~\ref{prop:from_one_to_many}, where the separation distance is identified as the tail distribution of the maximum of a set of independent random variables $T_n^i$, it is reasonable to ask how the above results relate to the theory of extreme values. If the random variables $\{T_n^i \}$ are i.i.d. for all $i$ and $n$ then the Fisher-–Tippet-–Gnedenko Theorem guarantees convergence in distribution of a renormalized $T_n$ to one of three possible distributions. For example, if $\Xn$ is a random walk on \hyp\ for which the rate of each coordinate is chosen at random, with 
\[ \prob{\rho_n^i = p_k} = q_k\,, \qquad k=1,\dots,m, \]
for all $i$ and $n$, Theorem 2.7.2 of \cite{Galambos-1978} shows that a renormalized $\Tn$ has a limiting Gumbel distribution. Indeed, writing $p^* = \min\set{p_k}$, direct calculation 
shows that
\begin{align*}
\separation_n \left(\frac{\log n+c}{2p^*}\right) &= 1-\left(1-\sum_{j=1}^m q_k \squ{\frac{e^{-c}}{n}}^{p_k/p^*}\right)^n \\
&\sim 1- \left(1-q^* \frac{e^{-c}}{n}\right)^n \xrightarrow[n\rinf]{} 1- \exp(-q^*e^{-c}) \,.
\end{align*}
In this case we see that both right- and left-hand windows are $\Ord(1)$.

More generally, the function $\theta_n$ defined in
equation~\eqref{eqn:defn-theta} may be interpreted as
follows. Let $\{V_n^i \,:\, 1\leq i\leq n\}$ be independent, identically
distributed random variables, whose distribution is a mixture over $\lambda$ of
$\expdist(\lambda)$ distributions, with mixture probability
distribution $\mu_n$. Then, for $t\geq0$,
\[ \prob{V_n^i>t} = \int_0^\infty e^{-\lambda t} \mu_n(d\lambda) \,,
\]
and so
\[ \Ex{\sum_{i=1}^n \indev{V_n^i>t}} = n\int_0^\infty e^{-\lambda t}
\mu_n(d\lambda) = \theta_n(t) \,. \] Thus $\theta_n(t)$ describes the
mean number of exceedances of level $t$ by the set of random variables
$\{V_n^i\}$. In particular, Proposition~\ref{prop:theta-behaviour} implies that the set of $n$-tuples driven by
$\mu_n$ exhibits a $\tau_n$-separation cutoff if and only if
\[ \Ex{\sum_{i=1}^n \indev{V_n^i>c\tau_n}} \xrightarrow[n\rinf]{}
\begin{cases}
\infty &\quad 0<c<1 \\
0 &\quad c>1 \,.
\end{cases}
\]

\medskip
\section{Coupling cutoffs}\label{sec:coupling cutoff}

It is well known that the coupling method can be used to bound the rate of convergence to equilibrium of a Markov chain, via the coupling inequality (see \cite{Lindvall-2002}). 
Let $\Xn$ and $\Yn$ be two copies of a Markov process on $\En$ with equilibrium distribution $\pin$.
\begin{defn}\label{defn:coupling}
        A \emph{coupling} of $\Xn$ and $\Yn$ is a process $(\hat{X}_n,\hat{Y}_n)$ on $\En\times \En$ such that
            \[ \hat{X}_n\stackrel{\mathcal{D}}{=} \Xn \quad \text{and} \quad \hat{Y}_n\stackrel{\mathcal{D}}{=} \Yn \,, \]
where $\stackrel{\mathcal{D}}{=}$ denotes equality in distribution.

        The \emph{coupling time} $\Tn^c$ of $\hat{X}_n$ and $\hat{Y}_n$ is defined by
            \[ \Tn^c =\inf\set{t\geq 0\,:\, \hat{X}_n^t = \hat{Y}_n^t} \,.\]
    \end{defn}

For a given coupling of $\Xn$ and $\Yn$, define
\begin{equation}\label{eqn:defn-F-coupling cutoff}
\bar{F}_n(t) = \prob{\Tn^c > t}, \quad t\geq 0 \,,
\end{equation}
to be the tail probability of $\Tn^c$. Suppose now that $\Xn^0=\xn^0$ is fixed, and that $\Yn^0\sim \pin^0$. We then define the following behaviour, in analogy with Definition~\ref{defn:cutoff}:

\begin{defn}\label{defn:coupling cutoff}
For $n\geq 1$, let $\Tn^c$ and $\bar{F}_n$ be defined as above. We say that
the sequence $\set{\En,\Xn,\pin,\Tn^c}$ exhibits a \emph{$(\tau_n,b_n)$-coupling-cutoff} if $\tau_n,b_n>0$ satisfy
$b_n=o(\tau_n)$ and
\begin{align}
\bar{F}_-(c) = \liminf_{n\rinf} \bar{F}_n(\tau_n+cb_n) \quad&\text{satisfies
$\lim_{c\rminf}\bar{F}_-(c) = 1$} \,, \label{eqn:F-inf-conv}\\
\bar{F}_+(c) = \limsup_{n\rinf} \bar{F}_n(\tau_n+cb_n) \quad&\text{satisfies
$\lim_{c\rinf}\bar{F}_+(c) = 0$} \,. \label{eqn:F-sup-conv}
\end{align}
\end{defn}

 Thus a coupling cutoff occurs when the distance between the two processes, measured using the tail probability of the coupling time $\Tn^c$, asymptotically exhibits an abrupt change from one to zero at time $\tau_n$. (Note that if $\Tn^c$ is a maximal coupling time for all $n$ \cite{Griffeath-1975} then a coupling-cutoff is equivalent to a total-variation cutoff.) As with the optimal SST of Section~\ref{sec:sep-cutoff}, if $(\Xn,\Yn)$ is a pair of $n$-tuples  whose $i^{th}$ coordinates may be independently coupled at an exponential rate $\lambda_n^i$, then $\Tn^c$ is the maximum of $n$ coupling times and this yields an analogous version of Theorem~\ref{thm:main-prelim} for coupling cutoffs.

For the random walks on \hyp\ considered in Section~\ref{ssec:hypercube}, no intuitive maximal coupling is known in general; for the \emph{symmetric} random walk a (nearly) maximal solution is presented in~\cite{Matthews-1987}, and a stochastically optimal \emph{co-adapted} coupling is described in~\cite{Connor.Jacka-2008}. However, $\Xn$ and $\Yn$ may be simply coupled 
by allowing their $i^{th}$ coordinates to evolve independently until the time that they first agree, whereafter they move synchronously. If
$\Xn^0$ and $\Yn^0$ do not agree on the $i^{th}$ coordinate
(which happens with probability $1/2$), then it follows that the time taken for
agreement on this coordinate is equal to the time of the first
incident on a Poisson process of rate $2p_n^i$, and so this coupling takes place at an exponential rate. Thus a random walk on \hyp\ exhibits a coupling cutoff if and only if it exhibits a separation cutoff (with the same values of $\tau_n$ and $b_n$).


In general, the assumption that each component of the $n$-tuples may be co-adaptedly coupled at an exponential rate is not restrictive: indeed, this is a reasonable assumption for many Markov processes of interest \cite{Burdzy.Kendall-2000}. 
There is also a possibility that the coupling variant of Theorem~\ref{thm:main-prelim} outlined above could have interesting consequences for coupling-based perfect simulation algorithms (such as CFTP and variants) for high-dimensional distributions.


\section*{Acknowledgements}
Thanks to Persi Diaconis for making time to discuss this topic during a visit by the author to Universit\'e Nice Sophia Antipolis in June 2007, and to Pierre Del Moral for making this visit possible. Thanks also to Wilfrid Kendall for motivating and improving the work in this paper while the author was working at the University of Warwick, and to a helpful referee.


\bibliographystyle{amsplain}
\bibliography{hypercube2}

\end{document}